\documentclass[12pt,a4paper]{scrartcl}
\usepackage[english]{babel}
\usepackage{setspace}
\usepackage{lscape}
\usepackage{amsmath}
\usepackage{amssymb}
\usepackage{graphicx}
\usepackage{amsthm}
\usepackage{mathrsfs}
\usepackage{filecontents}
\usepackage{hyperref}
\usepackage{color}
\usepackage[top=2.75cm,bottom=2.50cm,left=3.00cm,right=2.50cm]{geometry}
\usepackage{mathtools}
 \usepackage{bbm}
\usepackage[T1]{fontenc}
\usepackage{multirow}
\usepackage{ wasysym }
\usepackage{authblk}
  \newenvironment{proof*}[1]
 {\renewcommand{\proofname}{#1}%
   \begin{proof}}
 \makeatletter
\newenvironment{proofof}[1]{\par
  \pushQED{\qed}%
  \normalfont \topsep6\p@\@plus6\p@\relax
  \trivlist
  \item[\hskip\labelsep
        \itshape
    Proof of #1\@addpunct{.}]\ignorespaces
}{%
  \popQED\endtrivlist\@endpefalse
}
\makeatother

\def\({\left(}
\def\){\right)}
\newcommand{\defeq}{\coloneqq}
\newcommand{\eqdef}{\eqqcolon}
   
\newcommand{\R}{\mathbb R}   
   
\DeclareMathOperator{\sgn}{sgn}
\theoremstyle{plain}% default
\newtheorem{thm}{Theorem}
\newtheorem{prop}[thm]{Proposition}
\newtheorem{corr}[thm]{Corollary}
\newtheorem{lem}[thm]{Lemma}

\theoremstyle{definition}

\newtheorem{rem}[thm]{Remark}
\theoremstyle{definition}
\newtheorem{Def}[thm]{Definition}

\title{Geometry and topology of the Kerr photon region in the phase space}
\author[1]{Carla Cederbaum}
\author[1]{Sophia Jahns}
\affil[1]{Department of Mathematics, University of T\"ubingen\\ 

 Auf der Morgenstelle 10, 72076 T\"ubingen, Germany}
\date{}     
\begin{document}
\maketitle
%\doublespacing
\section{Abstract}

We study the set of trapped photons of a subcritical ($|a|<M$) Kerr spacetime as a subset of the phase space. First, we present an explicit proof that the photons of constant Boyer--Lindquist coordinate radius are the only photons in the Kerr exterior region that are trapped in the sense that 
they stay away both from the horizon and from spacelike infinity. 

We then proceed to identify the set of trapped photons as a subset of the (co-)tangent bundle of the subcritical Kerr 
spacetime.  We give a new proof showing 
that this set is a smooth $5$-dimensional submanifold of the (co-)tangent bundle with topology $SO(3)\times\mathbb R^2$ using results about the classification of $3$-manifolds.

Both results are covered by the rigorous analysis of Dyatlov~\cite{dyatlov}; however, the methods we use are very different and shed new light on the results and possible applications.

\section{Introduction}\label{intro0}

It is well-known that in the Schwarzschild spacetime with positive mass $M$, there is a family of photons that orbit the center of the spacetime at the fixed radius $r=3M$ (in Schwarzschild coordinates),
and the surface $\{r=3M\}$ which they fill is totally umbilical and called the \emph{photon sphere}. 
 
A similar phenomenon is known in the exterior of the subcritical Kerr spacetime: there is a $2$-parameter family of photons which stay at fixed coordinate radius~\cite{teo}. We give a thorough, explicit proof that these are the only trapped photons
in this spacetime family (Proposition~\ref{onlytrap}). This fact is often tacitly assumed to be well-known and indeed follows from Dyatlov's rigorous analysis~\cite{dyatlov}; whereas our proof explicitly demonstrates that the conventional computations (see e.g.~\cite{teo}) and heuristic arguments indeed do not overlook any trapped photons. 

Unlike in the Schwarzschild situation, the spacetime region in a subcritical Kerr spacetime where trapped photons are found (the \emph{region accessible to trapping} or the \emph{photon region}) is not a submanifold (with or without boundary) of the spacetime. 
Therefore, the spacetime itself seems not the right setting to investigate geometric properties of photon regions; instead, the (co-)tangent bundle, where the Kerr photon region will be seen to be a submanifold, 
is a much better setting for further investigation of its geometry. Of these two bundles, the cotangent bundle seems a more appropriate structure to understand photon regions,
since the natural symplectic form it carries allows to understand constants of motion as symmetries of the phase space (see e.g.~\cite{sommers}); 
and the analysis of the constants of motion plays an important role in understanding trapped light. 

For the Schwarzschild spacetime of positive mass, a direct computation shows that the set of trapped photons in the cotangent bundle is a submanifold of topology $SO(3)\times \R^2$.  The fact that the photon region in the Kerr phase space is a manifold of topology $ SO(3)\times \mathbb R^2$ as well was shown earlier as a consequence of more general results by Dyatlov in~\cite{dyatlov}, 
where the implicit function theorem was used on a family of slowly rotating Kerr spacetimes considering them as perturbations of the Schwarzschild spacetime. Our proof uses a different, more direct approach, which does not rely on knowledge about the Schwarzschild case and might help gain better insights into why the set of trapped photons in the Kerr phase space exhibits the properties in question. 
Knowing this alternative way to determine the topology might also be useful in possibly proving a uniqueness theorem for asymptotically flat, stationary, rotating, vacuum spacetimes with a photon region, in the spirit of e.g.~\cite{carla, carlagregpmt}, see below. 

Other than being interesting in its own right, understanding the geometry of photon regions is also relevant in a variety of contexts: light trapping phenomena play a role in the analysis of stability of black holes (see e.g. the lecture notes~\cite{dafermos}), 
as well as in gravitational lensing (see e.g.\,the review~\cite{Perlick2004}). 
Of current interest is also the role that photon regions play in the understanding of black hole shadows (see e.g.~\cite{grenze}). 

For static spacetimes, trapping of light has turned out to be a useful tool to prove uniqueness theorems; 
a static, asymptotically flat vacuum spacetime with a photon sphere as its inner boundary has to be a Schwarzschild spacetime \cite{carla, carlagregpmt};  analogous results holds for some  cases with matter
\cite{carlagregelectro, yaza, yaza2, yaza3}, and for one case even perturbative uniqueness was established~\cite{yoshino}, see also~\cite{shoom1, shoom2}.

Our article is organized as follows: In Section~\ref{intro}, we recall some well-known facts about the Kerr family and introduce some notation. 
In Section~\ref{sectiononlytrap}, we give a precise definition of trapping of light for stationary spacetimes and show that only photons of constant Boyer--Lindquist radius can be trapped in a subcritical Kerr spacetime. 
As a consequence of this, there are no null geodesically complete timelike umbilical hypersurfaces in the domain of outer communication in subcritical Kerr (Corollary~\ref{corUmbilical}).

The subsequent Section~\ref{inphase} explains how the set of trapped light can be understood as a subset of the (co-)tangent bundle; we prove that the photon region in the Kerr (co-)tangent 
bundle is a submanifold (Theorem~\ref{submanifold}). To this end, we make use of the characterization of trapped photons in terms of constants of motion and apply the implicit function theorem twice. 
Theorem~\ref{topo} shows that the photon region in the phase space of a subcritical Kerr spacetime has topology $SO(3)\times\mathbb R^2$. 
This is done by calculating its fundamental group via the Seifert--van Kampen theorem and using the classification of $3$-manifolds.

\section{Basic Facts and Notation}\label{intro}
We describe the Kerr spacetime of mass $M$ and angular momentum $a$ in Boyer--Lindquist coordinates $(t,r,\vartheta, \varphi)$ with $t\in\R$, radius $r>0$ suitably large, latitude $0\leq\vartheta\leq \pi$, and longitude $0\leq \varphi\leq 2\pi$. We use the following abbreviations: 
\begin{align*}
S&\defeq\sin\vartheta\\
C&\defeq\cos\vartheta\\
\rho^2&\defeq r^2+a^2\cos^2\vartheta\\
\Delta&\defeq r^2-2Mr+a^2\\
\mathscr A &\defeq\left(r^2+a^2\right)^2 -\Delta a^2\sin^2\vartheta.
\end{align*}

The metric of the Kerr spacetime in Boyer--Lindquist coordinates is given by
\begin{equation*}
-\left(1-\frac{2Mr}{\rho^2}\right)dt^2+\frac{\rho^2}{\Delta}dr^2+ \rho^2 d\vartheta^2 -\frac{4MraS^2}{\rho^2}d t d\varphi+\frac{\mathscr A}{\rho^2}S^2 d\varphi^2.
 \end{equation*}
 The \emph{domain of outer communication} (DOC) is the spacetime patch where $r>M+\sqrt{M^2-a^2}$.
The Boyer--Lindquist coordinates cover the entire DOC, except the \emph{axis} $\{S=0\}$. The breakdown of the metric in this form is a mere coordinate artefact; it can be extended smoothly over the axis, 
and so can the Killing vector fields $\partial_t$ and $\partial_\varphi$ (see \cite{oneill}).
We treat the \emph{subcritical} case $|a|<M$. In this case, the \emph{horizon} $\{\Delta=0\}$ has two connected components, and its outer component $\{r=M+\sqrt{M^2-a^2}\}$ is the boundary of the DOC.

Furthermore, we assume that $a\geq 0$. This is no restriction of generality, since a sign change of $a$ just means a change of the direction of rotation. 
If $a=0$, we recover the Schwarzschild metric. 

The Kerr spacetime will be denoted by $\left(K,g\right)$. \\

It is well-known that the motion of regular (that is, not entirely contained in the axis) photons in the DOC of Kerr is governed by the following equations of motion (see e.g.\;\cite{oneill}):
\begin{align}
\Delta\rho^2\dot t&=\mathscr A E-2MraL,\\
\rho^4\dot{r}^2&=E^2r^4+(a^2E^2-L^2-\mathfrak Q)r^2+2M((aE-L)^2+\mathfrak Q)r-a^2\mathfrak Q\eqdef R(r), \label{R}\\
\rho^4\dot{\vartheta}^2&=\mathfrak Q-\left(\frac{L^2}{S^2}-E^2a^2\right)C^2\eqdef{\Theta(\vartheta)},\label{Theta}\\
\Delta\rho^2\dot \varphi&=2MraE-\left(\rho^2-2Mr\right)\frac{L}{S^2}.
\end{align}
The dot denotes the derivative with respect to the affine parameter.

The quantity $E=-\langle \dot\gamma, \partial_t\rangle$ is the \emph{energy} of a geodesic $\gamma$, $L=\langle \dot\gamma, \partial_\varphi\rangle$ its \emph{angular momentum}, and $\mathfrak Q$ its 
 \emph{Carter constant}, given as $T\left(\dot\gamma, \dot \gamma\right)$, where $T$ is the Killing tensor given ba $T^{\mu \nu} =\frac{1}{r^2}g^{\mu\nu}+ 2 \rho^2 l^{( \mu}n^{\nu )}$ 
 (where $l^\mu=\frac{1}{\Delta}\left(r^2+a^2, 1, 0, a\right)$ and $n^\nu=\frac{1}{2\rho^2}\left(r^2+a^2, -\Delta, 0,a\right)$ are the \emph{principal null directions}).
 
 For dealing with general geodesics, it is useful to further introduce the quantity $\mathrm q\defeq g(\dot\gamma, \dot\gamma)$, which equals zero for null geodesics (photons). 
 
 These so-called \emph{constants of motion} $\mathrm q, E, L, \mathfrak Q$ are constant along each geodesic, since they are derived from Killing vectors and Killing tensors, respectively. 
 Note that they extend over the axis, since the generating tensors do so (see \cite{oneill}).

Wherever $\partial_t$ is timelike, all photons have positive energy $E>0$.
In case $E\neq 0$, we can rewrite the so-called \emph{$R$-equation} \eqref{R} and \emph{$\Theta$-equation} \eqref{Theta} as the \emph{scaled} \emph{$R$-equation}
and the \emph{scaled} \emph{$\Theta$-equation}, respectively:
\begin{align}
& \left(\frac{\rho^2}{E}\right)^2\dot{r}^2=r^4+(a^2-\Phi^2-Q)r^2+2M((a-\Phi)^2+Q)r-a^2Q\left(=\tfrac R {E^2}\right), \label{Rscale}\\
 &\left(\frac{\rho^2}{E}\right)^2\dot{C}^2=Q-(Q+\Phi^2-a^2)C^2-a^2C^4 \left(= \tfrac \Theta {E^2}\right),\label{Thetascale}
\end{align}

with the \emph{conserved quotients} $\Phi\defeq \frac{L}{E}$, and  $Q\defeq \frac{\mathfrak Q}{E^2}$ (see \cite{teo}).

In the Schwarzschild case $a=0$, the photons in the DOC with constant $r$-coordinate are precisely those that fulfill $r(s)=3M$ and $\dot r(s)=0$ for all parameter values $s\in\R$. 
The set $\{r=3M\}$ in Schwarzschild is called the \emph{photon sphere}. 

It was shown in \cite{teo} that in the DOC of a subcritical Kerr spacetime with $a\neq 0$ all photons with constant $r$-coordinate (\emph{spherical photon orbits}) belong to the one-parameter class of solutions of 
$R(r)=\frac{\partial R(r)}{\partial r}=0$ given by

\begin{align}
 \Phi_{\text{trap}}(r)&\defeq-\frac{r^3-3Mr^2+a^2r+a^2M}{a(r-M)},\label{teo} \\ 
Q_{\text{trap}}(r)&\defeq-\frac{r^3(r^3-6Mr^2+9M^2r-4a^2M)}{a^2(r-M)^2}\label{teoQ}
\end{align}
(see~\eqref{R} and~\eqref{Rscale}), where $r$ ranges from \begin{equation*} \hat{r}_1\defeq 2M(1+\cos(2/3 \arccos(-a/M)))\end{equation*} to \begin{equation*}\hat{r}_2\defeq 2M(1+\cos(2/3 \arccos(a/M))).\end{equation*}

The \emph{corotating} (\emph{counterrotating}) orbits (i.e., those with positive (negative) angular momentum $L$) are located inside (outside) the hypersurface $\{r=r_m\}$, where 

\begin{equation}\label{rm}
r_m\defeq M+2\sqrt{M^2-\frac 1 3 a^2}\cos\left(\frac 1 3 \arccos \frac{M(M^2-a^2)}{(M^2-\frac 1 3 a^2)^{\frac 3 2 }}\right),
\end{equation}
while those at the intermediate radius $r_m$ have vanishing angular momentum $L$. 

The function $\Phi_{\text{trap}}: \left[\hat r_1, \hat r_2 \right]\rightarrow \mathbb R$ is strictly monotonically decreasing and has only one zero, located at $r_m$, 
while $Q_{\text{trap}}: \left[\hat r_1, \hat r_2 \right]\rightarrow \mathbb R_{\geq 0}$ has its zeros at $\hat r_1$ and $\hat r_2$, is strictly increasing on $\left[\hat r_1, 3M\right]$, 
and strictly decreasing on $\left[3M, \hat r_2\right]$. The quantity $\Phi_{\text{trap}}$ determines the maximal latitude for trapped photons. 
In particular, trapped photons with radius $r$ can only reach the the axis $\{S^2=0\}$ if  $\Phi_{\text{trap}}(r)=0$. 

For these facts and more information about \emph{spherical photon orbits} (i.e., those of constant radius) in the Kerr spacetime, see \cite{teo}.

 \section{Trapped Photons in the Kerr spacetime}\label{sectiononlytrap}

 It is natural to ask whether there are more photons than those of constant radial coordinate which are ``trapped'' around the Kerr center, without ever falling through the horizon or escaping to spatial infinity. 
 Before answering this question, we first need a more precise notion of what it means to be trapped.
 
 We suggest the following definition for the rather general case of stationary spacetimes: 
\begin{Def}
A photon in the DOC of a stationary spacetime is called \emph{trapped} if its orbit in the quotient of the DOC under the action of the stationary Killing vector field $\partial_t$ is contained in a compact set. 
\end{Def}

Thus, in the Kerr spacetime, a photon is trapped if and only if the range of its radial coordinate is a relatively compact set contained in $\left(M+\sqrt{M^2-a^2}, \infty\right)$. 

\begin{prop}\label{onlytrap}
 The spherical photons are the only trapped photons in the DOC of a subcritical ($0<a<M$) Kerr spacetime.
\end{prop}

\begin{proof}
Let $\gamma=(t,r,\vartheta, \varphi)$ be a trapped photon of non-constant radial coordinate in the DOC of a subcritical Kerr spacetime. 
We will first treat the \textbf{case that $\gamma$ does not intersect the axis $\{S^2=0\}$}. 

We focus on the \textbf{subcase that the photon has nonvanishing energy $E\neq 0$} and investigate the right-hand side $R(r)$ of Equation \eqref{R}. 
Since its left-hand side is manifestly non-negative, $\gamma$ has to lie in a region
of the Kerr space-time where $R(r)$ is non-negative. 

Since the range of the trapped photon's radial coordinate is a relatively compact set contained in $\left(M+\sqrt{M^2-a^2}, \infty\right)$,
there must be two values $M+\sqrt{M^2-a^2}<r_1 < r_2<\infty$ such that the photon either turns around when reaching $r_i$ (that is, the sign of $\dot r$ changes at some $s_i$ with $r(s_i)=r_i$) or asymptotically approaches $r_i$ 
(that is, $\lim\limits_{s\to +\infty}r(s)=r_i$ or $\lim\limits_{s\to-\infty}r(s)=r_i$). 

As a consequence, for trapping with non-constant radial component, we need two zeros $r_i$ ($i=1,2$) of $R(r)$ that lie in the interval $\left(M+\sqrt{M^2-a^2}, \infty\right)$, with $R(r)\geq 0$ between these zeros. 

Thus, on the search for trapped photons of non-constant coordinate radius, analyzing the polynomial $R(r)$ of degree $4$ with leading coefficient $E^2>0$, one would have to find constants of motion such that all roots of $R$ are real, and 
(writing $r_0\leq r_1\leq r_2\leq r_4$ for these roots) such that
 $r_1,r_2, r_3$ are outside the outer horizon component $\{r=M+\sqrt{M^2-a^2}\}$, and $r_1$ and $r_2$ do not coincide, see Figure~\ref{Rplot}.

\begin{figure}[h]
\centering
\scalebox{.5}{\includegraphics{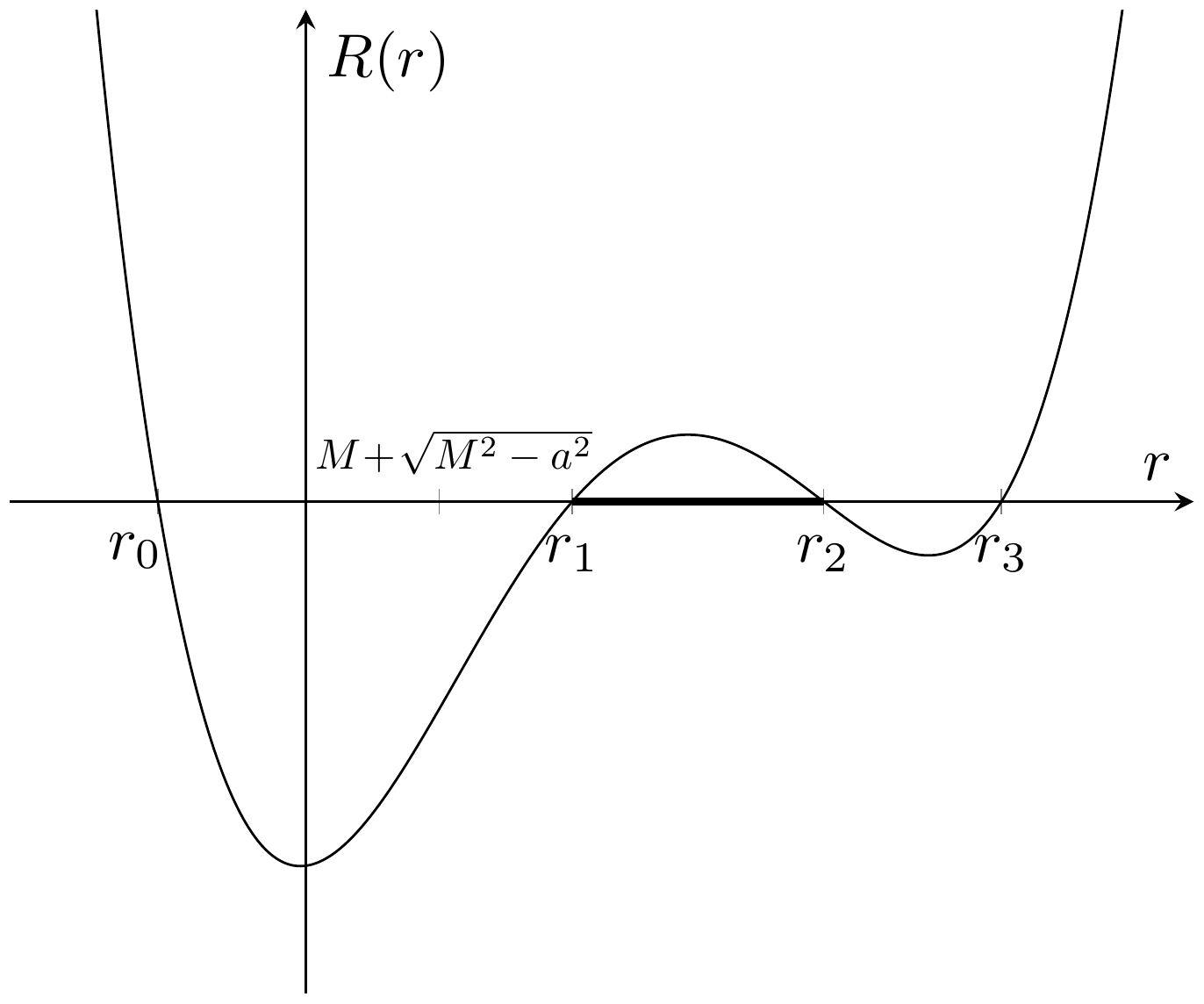}}
\caption{A qualitative plot of $R$ against $r$ for a hypothetical trapped photon of non-constant coordinate radius.}
\label{Rplot}
\end{figure}

Since we assumed that $E\neq 0$, we may work with the scaled $R$-equation \eqref{Rscale}. 
We see immediately that the coefficient of $r^2$ has to be negative to allow for the constellation of roots we need (since otherwise $\frac{d}{dr}\left(\frac{R}{E^2}\right)$ would be monotonous and $\frac{R}{E^2}$ convex), that is,
 \begin{equation}\label{aphiq}
           a^2-\Phi^2-Q<0.
  \end{equation}
 
 On the other hand, since the left-hand side of the scaled $\Theta$-equation is manifestly non-negative, choosing a non-positive $Q\leq0$ forces $a^2-\Phi^2-Q\geq0$, contradicting 
 \eqref{aphiq}.
 We have thus ruled out the case $Q\leq0$ 
 for trapped light and may focus on the case $Q> 0$. 
 
Since $Q>0$ implies $\frac{R(0)}{E^2}<0$, a necessary condition for the second root $r_1$ of $\frac{R(r)}{E^2}$ to be outside the outer horizon component is $\frac{R(M)}{E^2}<0$.
But on the other hand, we can estimate (using $Q>0$ and $a<M$):
\begin{equation*} 
 \frac{R(M)}{E^2}=M^2((\Phi-2a)^2+M^2-a^2)+M^2Q-a^2Q\geq M^2(\Phi-2a)^2\geq 0,
\end{equation*}
thus ruling out the existence of trapped light with $r\neq const.$ outside the axis with nonvanishing energy $E$.

In the \textbf{subcase of vanishing energy $E=0$}, the $R$-equation reduces to 
\begin{equation*}
\rho^4\dot{r}^2=-(L^2+\mathfrak Q)r^2+2M(L^2+\mathfrak Q)r-a^2\mathfrak Q.
\end{equation*}
The roots of the right-hand side of this equation are given by
\begin{equation*}
r_{1,2}=M\pm\sqrt{M^2-\frac{a^2\mathfrak Q}{L^2+\mathfrak Q}},
\end{equation*}
and at least one of them (if real) is smaller than the radius of the outer horizon component, $M+\sqrt{M^2-a^2}$.

Summing up, so far we have shown that there are no trapped photons contained in $K\setminus \{S^2=0\}$ with nonconstant radial component. 

We need to treat the \textbf{case of photons that are contained in the axis $\{S^2=0\}$} separately. The axis is a $2$-dimensional, totally geodesic submanifold with line element 
\begin{equation*}
-(1-2Mr/\rho^2)dt^2+(\rho^2/\Delta) dr^2,
\end{equation*}
so every photon contained in the axis has to fulfill
\begin{equation*}
-(1-2Mr/\rho^2)\dot{t}^2+(\rho^2/\Delta) \dot{r}^2=0. 
\end{equation*}
For trapped photons, $\dot{r}$ has to vanish at some parameter value (at least asymptotically), while $\dot t\neq 0$ cannot approach $0$. This gives $(1-2Mr/\rho^2)=0$, that is, $\Delta=0$, which means that the photon is not in the DOC. 

The \textbf{case of photons that cross the axis but are not entirely contained in it} can be treated like the off-axis case: the $R$-equation is also valid on the axis, 
and the conditions on the conserved quotients that were derived from the scaled $\Theta$-equation still would have to be fulfilled, since the conserved quotients can be calculated off-axis. 
\end{proof}
 The statement of Proposition~\ref{onlytrap} was also discussed in \cite{claudio}, using different techniques, and follows from~\cite{dyatlov}. 

 We can use Proposition~\ref{onlytrap} to show nonexistence of certain umbilical hypersurfaces in the DOC. We remind the reader that a semi-Riemannian manifold is called \emph{null geodesically complete} if 
 every inextendible null geodesic is defined on all of $\mathbb R$. 

\begin{corr}\label{corUmbilical} There are no null geodesically complete, timelike umbilical hypersurfaces in the DOC of the subcritical Kerr spacetime (with $a\neq 0$) 
whose quotient under the stationary Killing vector field $\partial_t$ is contained in a compact set. In particular, none of the cylinders $\{r=const.\}$ are umbilical.
\end{corr}
\begin{proof}
Assume towards a contradiction that there is such a hypersurface $\mathscr S$. %In a first step we make the additional assumption that $\mathscr S$ is null geodesically complete. 
By Proposition 3 in \cite{perlick}, a timelike hypersurface in a Lorentzian manifold is umbilical if and only if every null geodesic that is initially tangent to it remains tangent throughout. 
Since by the assumption the quotient of $\mathscr S$ after factoring out $\partial_t$ is relatively compact in the quotient of the DOC, 
every photon that stays on $\mathscr S$ has the property that the range of its radial coordinate is a relatively compact set contained in $\left(M+\sqrt{M^2-a^2}, \infty\right)$. 
These two facts combine to the conclusion that every photon that is initially tangent to $\mathscr S$ is trapped. 

By Proposition~\ref{onlytrap}, the only trapped photons are those of constant radial coordinate, 
hence for any given radius where trapped photons exist, there is only one choice for their conserved quotients $\Phi$ and $Q$ 
(namely, the values for $\Phi_\text{trap}$ and $Q_\text{trap}$ from Equations \eqref{teo} and \eqref{teoQ}). 
We are free to choose a positive value of $E$ for a trapped photon; that is, there is one degree of freedom for choosing the constants of motion for a trapped photon at a given point. 
These constants of motion determine the off-axis trapped photons completely, possible up to a sign choice for $\dot\vartheta$. 
That is, at every point in the Kerr DOC, there is (at most) one degree of freedom for choosing an initial direction for a trapped photon.

On the other hand, at every point of $\mathscr S$, the cone of lightrays tangent to $\mathscr S$ is 2-dimensional, 
which means that there are two degrees of freedom for choosing trapped photons at a given point of $\mathscr S$. This gives a contradiction, proving that there is no such $\mathscr S$. 
\footnote{In Schwarzschild, the statement is no longer true, as is demonstrated by the striking counterexample of the photon sphere $\{r=3M\}$. 
The proof of Corollary~\ref{corUmbilical} cannot be immitated in the $a=0$ case since there we do not have Equations~\ref{teo} and~\ref{teoQ} at our disposition.}
\end{proof}

\section{The Kerr Photon Region as a Submanifold of the (Co-)tangent Bundle}\label{inphase}

We call the region in the Kerr spacetime where one can find trapped photons the \emph{region accessible to trapping}. 
In contrast to the Schwarzschild case, in the subcritical Kerr spacetime the region accessible to trapping is not a submanifold (with our without boundary) of the spacetime; 
it can be imagined as the region covered by a crescent moon that rotates about an axis through its pointy ends, see Figure~\ref{crescent}. 
\begin{figure}[h]
\centering
\scalebox{.5}{\includegraphics{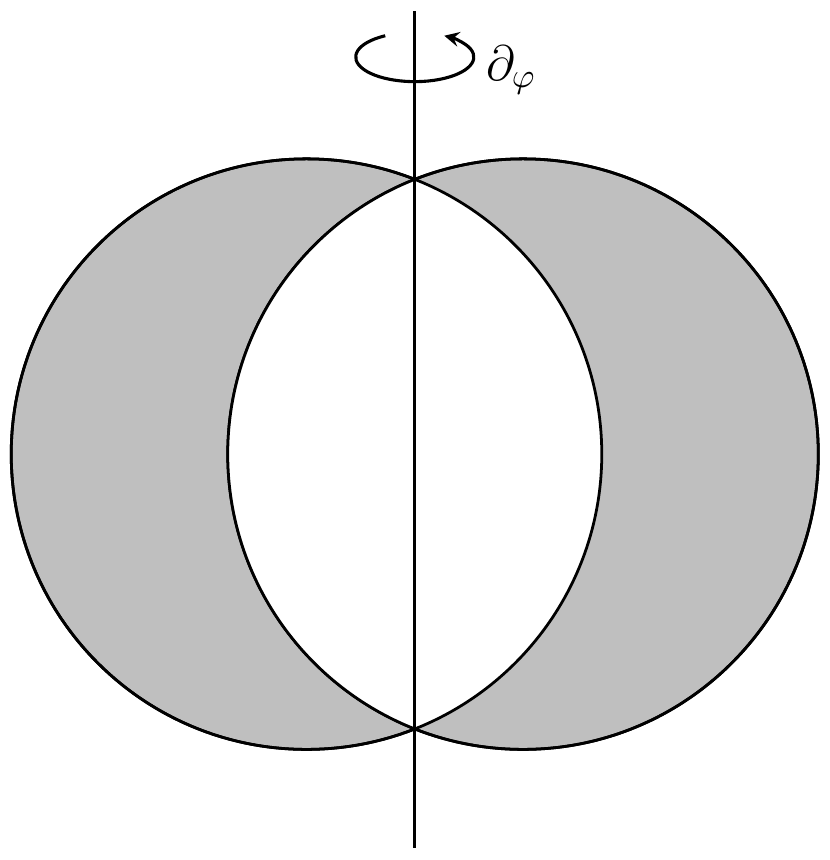}}
\caption{A $t=const.$ slice of the photon region in the Kerr spacetime.}
\label{crescent}
\end{figure}

We will identify the set of trapped photons in any spacetime with a subset of the \mbox{(co-)}tangent bundle. 
In the subcritical Kerr spacetime, it turns out that this subset of the (co-)tangent bundle is a much nicer geometrical object than the region accessible to trapping.

One can identify geodesics in any pseudo-Riemannian manifold $\left(M, g\right)$ with points in the tangent bundle $TM$  in the following way: 
The natural map 
\begin{equation*}          
\{\text{geodesics in }M\}\rightarrow TM
\end{equation*}

given by 
\begin{equation*}
 \gamma \mapsto \left(\gamma(0), \dot{\gamma}(0)\right)
\end{equation*}
is a bijection by uniqueness of geodesics. The canonical isomorphism $TM\rightarrow T^*M$ now gives a similar identification of geodesics with points in $T^*M$.

We only consider future-directed trapped photons; by reversing the direction of time (using the symmetry $t\mapsto -t$) one can extend all statements to past directed photons. 

The task in the present section is to show that the \emph{photon region in the tangent bundle} of the subcritical Kerr spacetime $P'\subseteq TK$ 
and the \emph{photon region in the cotangent bundle} $P\subseteq T^*K$ (that is, the image of $P'$ under the canonical isomorphism $TM\rightarrow T^*M$) are smooth submanifolds of $TK$ and $T^*K$, respectively. 
Of course, statements about the submanifold structure and the topology of $P'$ and $P$ imply each other immediately, but it will turn out to be useful to work in both settings.

\begin{rem} If $\left(M, g\right)$ is any spacetime, the above identification allows us to view the set of photons in $M$ as \begin{equation*}\{(m,p)\in TM: g_m(p,p)=0\}.\end{equation*} As the preimage of $0$ under $(m,p)\mapsto g_m(p,p)$, this set is a submanifold of $TM$ (and as regular as the spacetime itself), 
since the differential of this map contains the components of $2g_m(\cdot, p)$ as its last $n$ matrix components, and this never vanishes by non-degeneracy of $g$. 
\end{rem}

We can naturally extend all functions that were defined on the DOC of Kerr to the tangent bundle 
(of the DOC) of the Kerr spacetime, so that now in particular the constants of motion and the conserved quotients, but also  $\Phi_{\text{trap}}$ and $Q_{\text{trap}}$ are functions on $TK$.
Slightly abusing notation, we will use the same letters to denote these new functions. This remark applies---mutatis mutandis---also to $T^*K$.

The region accessible to trapping intersects the axis in only two points; justified by the equatorial symmetry $\vartheta\mapsto -\vartheta$, 
we will call one of them the North Pole and the other one the South Pole, and we will refer to the two connected components of the region accessible to trapping without the equatorial plane $\{S^2=1\}$ as the Northern and the Southern Hemisphere.

\begin{rem} As it turns out, in order to obtain later the topology of $P'$ and $P$ we will need to construct smooth bundle charts for a spatial slice of the image of $P$ under the canonical isomorphism $TK\rightarrow T^*K$. 
Nonetheless, this only proves that $P$ is a manifold, not that it is a submanifold of $T*K$. 
Even if one is only interested in the manifold structure of $P'$ and $P$, not their submanifold structure, 
the proof of the submanifold property has advantages over the explicit construction of charts, since the latter will turn out to be quite technical.  
\end{rem}

On $TK =K\times \mathbb R^4$, 
we use coordinates $(t,r,\vartheta, \varphi , p^0,p^1, p^2, p^3)$, where $( p^0,p^1, p^2, p^3)$ are the components of the contravariant tangent vector in the same coordinate basis. 

\begin{rem}
In the Schwarzschild case $a=0$, it is easy to see that the photon sphere lifts to a submanifold of the tangent bundle of Schwarzschild: this photon region consists precisely of all photons of the form
$(t,3M,\vartheta, \varphi, p^0, 0, p^2, p^3)$ and is thus in particular a submanifold of the space of all photons.  
\end{rem}

We shall now consider the situation in the subcritical Kerr spacetime with $a\neq 0$.

\begin{prop} \label{off-axis}
The Kerr photon region in the (co-)tangent bundle $TK$ ($T^*K$) of a subcritical Kerr spacetime is a smooth submanifold when the axis $\{S^2=0\}$ is deleted; that is, $P'\setminus \{S^2=0\}$ is a 
smooth submanifold of $TK$ (and $P\setminus \{S^2=0\}$ is a smooth submanifold of $T^*K$).
\end{prop}
\begin{proof}
We will prove the statement for $P'\subseteq TK$.

By the equations of motion and the conditions for trapped photons (see \cite{oneill} and \cite{teo}), 
a Kerr geodesic $\gamma =(t,r,\vartheta, \varphi , p^0,p^1, p^2, p^3)$ is a trapped photon if and only if 
\begin{align}
 p^0 &= \frac{E}{\Delta \rho^2} \left(\mathscr A -2Mra\Phi_{\text{trap}}\right), \label{4}\\
 p^1 &=0,\label{4.2}\\
 (p^2)^2 &=\frac{E^2}{\rho^4}\left(Q_{\text{trap}}-\left(\frac{\Phi_{\text{trap}}^2}{S^2}-a^2\right)C^2\right), \\
 p^3 & = \frac{E}{\Delta\rho^2}\left(2Mra+(\rho^2-2Mr)\frac{\Phi_{\text{trap}}}{S^2}\right)\label{4end},
\end{align}
for some energy $E=-g_{00}p^0-g_{03}p^3>0$, where $ Q_{\text{trap}}$ and $\Phi_{\text{trap}}$ are the functions of $r$ which give the trapping conditions from \cite{teo} 
stated in Equations \eqref{teo} and \eqref{teoQ}. 

Note that these equation are invariant under the scaling 
\begin{equation*}
\(E, p^0, p^1, p^2, p^3\)\mapsto  \lambda\(E, p^0, p^1, p^2, p^3\)
\end{equation*}
 with $\lambda>0$, which allows us to get rid of one of the above equations:

Solving Equation~\eqref{4} for 
\begin{equation*}
e\defeq E( r, \vartheta, p^0)=\frac{p^0\Delta\rho^2}{\mathscr A-2Mra\Phi_{\text{trap}}}
\end{equation*}
and plugging it into Equations~\eqref{4.2}--\eqref{4end} yields 

\begin{align}
  p^1&=0 \label{3}\\
 (p^2)^2 &=\frac{e^2}{\rho^4}\left(Q_{\text{trap}}-\left(\frac{\Phi_{\text{trap}}^2}{S^2}-a^2\right)C^2\right), \label{3mid}\\
 p^3 &=  \frac{e}{\Delta\rho^2}\left(2Mra+(\rho^2-2Mr)\frac{\Phi_{\text{trap}}}{S^2}\right).\label{3end}
\end{align}

Hence, the off-axis photon region in the tangent bundle is the preimage of $0$ under the following smooth function $f=(f_1, f_2, f_3): TK\setminus\{S^2=0\}\rightarrow\mathbb R$,
\begin{align}
 f_1&\defeq p^1,\\
 f_2 &\defeq (p^2)^2 -\frac{e^2}{\rho^4}\left(Q_{\text{trap}}-\left(\frac{\Phi_{\text{trap}}^2}{S^2}-a^2\right)C^2\right), \label{e} \\
 f_3 & \defeq p^3 - \frac{e}{\Delta\rho^2}\left(2Mra+(\rho^2-2Mr)\frac{\Phi_{\text{trap}}}{S^2}\right).
\end{align}
We will show that the differential of $f$ has full rank, so that we can use the submersion theorem. 

Some partial derivatives of $f$ are: 
\begin{align*}
%\frac{\partial f_1}{\partial p^0} &=0 \phantom{tatata} 
\frac{\partial f_1}{\partial p^1}&=1 \phantom{tatata} \frac{\partial f_1}{\partial p^3}=0\\
%\frac{\partial f_2}{\partial p^0}&= -\left(\frac{2e\Delta}{\mathscr A\rho^4}\right)\left(Q_{\text{trap}}-\left(\frac{\Phi_{\text{trap}}^2}{S^2}-a^2\right)C^2\right)\\
\frac{\partial f_2}{\partial p^1} &=0 \phantom{tatata} \frac{\partial f_2}{\partial p^2}=2p^2 \phantom{tatata} \frac{\partial f_2}{\partial p^3}=0\\
\frac{\partial f_2}{\partial r} &= \frac{\partial }{\partial r}\left(\frac{e^2}{\rho^4} \right)\left(Q_{\text{trap}}
-\left(\frac{\Phi_{\text{trap}}^2}{S^2}-a^2\right)C^2\right)+\frac{e^2}{\rho^4}\left(\frac{\partial Q_{\text{trap}}}{\partial r}-\frac{\partial \Phi^2_{\text{trap}}}{\partial r}\frac{C^2}{S^2}\right)\\
\frac{\partial f_2}{\partial \vartheta} &= \frac{\partial }{\partial \vartheta}\left(\frac{e^2}{\rho^4} \right)\left(Q_{\text{trap}}
-\left(\frac{\Phi_{\text{trap}}^2}{S^2}-a^2\right)C^2\right)-\frac{e^2}{\rho^4}2C\left(\frac{\Phi_{\text{trap}}^2}{S^3}+\left(\frac{\Phi_{\text{trap}}^2}{S^2}-a^2\right)S\right)
\\
%\frac{\partial f_3}{\partial p^0} &=0 \phantom{tatata} 
\frac{\partial f_3}{\partial p^3}&=1
\end{align*}
 
By the above formulas for the partial derivatives of $f_1$ and $f_3$ we see that the differential of $f$ has at least rank 2. In order to show that $f$ is indeed submersive, 
we assume (towards a contradiction) that the differential of $f$ has rank $2$ at some point of the off-axis photon region $P'\setminus \{S^2=0\}$ in $TK$. 
Then $\frac{\partial f_2}{\partial p^2}$, $\frac{\partial f_2}{\partial r}$, and $\frac{\partial f_2}{\partial \vartheta}$ vanish at this point; and in particular
\begin{align*}
     \frac{\partial f_2}{\partial p^2}=2p^2=0
\end{align*}
implies 
\begin{align}
 \left(Q_{\text{trap}}-\left(\frac{\Phi_{\text{trap}}^2}{S^2}-a^2\right)C^2\right)=0\label{h1}
 \end{align}

because of Equation \eqref{e} and the fact that $e$ is non-zero at the photon region in $TK$.
Hence, $\frac{\partial f_2}{\partial r} =0$ and $\frac{\partial f_2}{\partial \vartheta} =0$ reduce to 
\begin{align}
\frac{\partial Q_{\text{trap}}}{\partial r}-\frac{\partial \Phi^2_{\text{trap}}}{\partial r}\frac{C^2}{S^2}=0,\label{h3}
 \end{align}
and 
\begin{align}
2C\left(\frac{\Phi_{\text{trap}}^2}{S^3}+\left(\frac{\Phi_{\text{trap}}^2}{S^2}-a^2\right)S\right)=0. \label{h2}
 \end{align}

In the \textbf{case $C\neq 0$}, combining Equations~\eqref{h1} and~\eqref{h2} yields

\begin{equation*}
 \Phi_{\text{trap}}^2+ Q_{\text{trap}}\frac{S^4}{C^2}=0. 
\end{equation*}
But since $Q_{\text{trap}}\geq 0$ in the region accessible to trapping, this implies that $\Phi_{\text{trap}}^2$ and $Q_{\text{trap}}$ vanish for the same value of $r$, which cannot happen (see Section~\ref{intro}).

In the \textbf{case $C = 0$}, Equations \eqref{h1} and \eqref{h3} give $Q_{\text{trap}}=0$ and $\frac{\partial Q_{\text{trap}}}{\partial r}=0$, which also cannot happen for the same radius $r$, see Section~\ref{intro}.

We have thus derived a contradiction from the assumption that the differential of $f$ does not have full rank somewhere on the off-axis photon region. 

Thus, $0$ is a regular value of $f$, and we may conclude by the submersion theorem that the photon region of the Kerr spacetime with deleted axis is a submanifold of the tangent bundle $TK$. 

\end{proof}

Due to the coordinate failure, we have to deal with the axis separately. To do this, we will once again use the submersion theorem, this time applying it to a different function $TK\rightarrow\mathbb R^3$. 

\begin{prop}\label{on-axis}
 For a subcritical Kerr spacetime $K$, there is a neighborhood of the North (resp.\;South) Pole of the photon region $P'$ which is a smooth submanifold of $TK$ 
(or equivalently, a neighborhood of the North (resp.\;South) Pole of the photon region $P$ which is a smooth submanifold of $T^*K$).
\end{prop}

\begin{proof}
 
 We consider the function $h:TK\rightarrow \mathbb R^3$ given by 
 \begin{equation*}
h=\left(h_1, h_2, h_3\right)\defeq(\mathrm q, \mathfrak{Q}-E^2 \cdot Q_{\text{trap}}, L  -E\cdot\Phi_{\text{trap}}), 
\end{equation*}

where $Q_{\text{trap}}$, $\Phi_{\text{trap}}$ are again the functions given in Equations \eqref{teo} and \eqref{teoQ}. 

Clearly, the upper hemisphere of the photon region can be described as \begin{equation*}\{\(m,p\)\in TK: p^0(\(m,p\))>0, h(\(m,p\))=0\}.\end{equation*} 

The task is to show that $h$ is a submersion in a neighborhood of $P'\cap\{S^2=0\}$ in $ TK$. For the calculations near the axis, we change coordinates 
\begin{equation*}(t,r,\vartheta, \varphi)\mapsto (t,r, x\defeq r \cos\varphi\sin\vartheta, y\defeq r \sin\varphi\sin\vartheta).\end{equation*}

  We can work with the new coordinates on each hemisphere; for simplicity let us only consider the Northern Hemisphere, which suffices by the equatorial symmetry. 
Since also $p^2, p^3$ fail on the axis, we change the $p^i$ accordingly $
(p^0, p^1, p^2, p^3)\mapsto (\widetilde{p}^0, \widetilde{p}^1, \widetilde{p}^2, \widetilde{p}^3)$, 
 where $\widetilde p^0, \widetilde{p}^1, \widetilde{p}^2, \widetilde{p}^3$ are the natural coordinates adapted to our new coordinate system on the Northern Hemisphere of $K$, i.e.,
\begin{align*}
 p^0 &=\widetilde p^0, \hspace{1cm} p^1=\widetilde{p}^1, \\  
 p^2&=-\frac{(x^2+y^2)^{\frac 1 2 }}{zr}\widetilde{p}^1+\frac{1}{z(x^2+y^2)^{\frac{1}{2}}}\left(x\widetilde{p}^2+y\widetilde{p}^3\right), \\
  p^3&=\frac{1}{x^2+y^2}\left(-y\widetilde{p}^2+x\widetilde{p}^3\right),
	\end{align*}
	using the abbreviation $z\defeq\sqrt{r^2-x^2-y^2}$.
	
	We note that \begin{align*}
p^2&=\mathcal O(1) \phantom{tatatatai}  \frac{\partial p^2}{\partial \widetilde p^2}=\mathcal O(1) \phantom{tatatatai}  \frac{\partial p^2}{\partial \widetilde p^3}=\mathcal O(1)\\
p^3&=\mathcal O(S^{-1})\phantom{tatata}  \frac{\partial p^3}{\partial \widetilde p^2}=\mathcal O(S^{-1}) \phantom{tatata}  \frac{\partial p^3}{\partial \widetilde p^3}=\mathcal O(S^{-1})
  \end{align*}
 
 as $S\rightarrow 0$.
 
Furthermore, the metric components with respect to the new coordinates are given on the axis as
\begin{equation*}
\widetilde g_{\alpha \beta}=diag \left(-\frac{\Delta}{r^2+a^2}, \frac{r^2+a^2}{\Delta}, \frac{r^2+a^2}{r^2}, \frac{r^2+a^2}{r^2}\right). 
  \end{equation*}

We calculate that on the axis
 \begin{align*}
 \frac{\partial \mathrm q}{\partial \widetilde p^0}  &=-2 \frac{\Delta}{r^2+a^2} \widetilde p^0,\\
  \frac{\partial \mathrm q}{\partial \widetilde p^2}& =2 \frac{r^2+a^2}{r^2} \widetilde p^2, \\
	  \frac{\partial \mathrm q}{\partial \widetilde p^3}& =2 \frac{r^2+a^2}{r^2} \widetilde p^3.
  \end{align*}

 For a function $G: TK\rightarrow \mathbb R$, we will use the notation $G=\mathcal O_{1, \widetilde p} (S)$ to subsume $G=\mathcal O(S)$ and $\frac{\partial G}{\partial \widetilde p^\alpha}=\mathcal O(S)$ for all $0\leq \alpha\leq 3$. 

To calculate some partial derivatives of the Carter constant
\begin{equation*}\mathfrak Q=r^2\mathrm q +\frac{1}{\Delta}\left((r^2+a^2)^2(p_0)^2-\Delta ^2 (p_1)^2+a^2 (p_3)^2 +2(r^2+a^2)ap_0p_3\right) \end{equation*}
on the axis, we first note that $p_0=\left(-1+\frac{2Mr}{\rho^2}\right)\widetilde p^0+ \mathcal O_{1, \widetilde p} (S)$, $p_1=\frac{\rho^2}{\Delta}\widetilde p^1$, 
and $p_3=\mathcal O_{1, \widetilde p} (S)$ as $S\to 0$. 

This gives 
\begin{equation*}
\mathfrak Q=r^2\mathrm q +\frac{1}{\Delta}\left(\left(r^2+a^2\right)^2\left(-1+\frac{2Mr}{\rho^2}\right)^2\left(\widetilde p^0\right)^2-\rho^4\left(\widetilde p^1\right)^2\right)+ \mathcal O_{1, \widetilde p} (S)
\end{equation*} 
as $S\to 0$, and hence 

\begin{align*}
  & \frac{\partial \mathfrak Q}{\partial \widetilde p^0}\left|_{S^2=0}  =  r^2 \frac{\partial \mathrm q}{\partial \widetilde p^0}\left|_{S^2=0}  \right. \right.
  +2\Delta \widetilde p^0  ,   \\ 
   & \frac{\partial \mathfrak Q}{\partial \widetilde p^\alpha}\left|_{S^2=0}  =  r^2 \frac{\partial \mathrm q}{\partial \widetilde p^\alpha}\left|_{S^2=0}  \ \text{for}\ \alpha=2,3.\right.\right.
\end{align*}

Similarly, we calculate

\begin{align*}
 L   =  \langle \partial_\varphi, \dot\gamma\rangle & = -2 \frac{MraS^2}{\rho^2}p^0+\left(r^2+a^2+\frac{2Mra^2S^2}{\rho^2}\right)S^2 p^3\\
  & = -2 \frac{MraS^2}{\rho^2}\widetilde p^0+\frac{1}{r^2}\left(r^2+a^2+\frac{2Mra^2S^2}{\rho^2}\right)  \left(-y\widetilde p^2 + x\widetilde p^3\right), 
\end{align*}
and get, using $\frac{\partial S^2}{\partial x}=\mathcal O(S)$ and  $\frac{\partial S^2}{\partial y}=\mathcal O(S)$:

\begin{align*}
         &  \frac{\partial L}{\partial \widetilde p^\alpha}\left|_{S^2=0}  =  0 \right.  \ \forall \ 0\leq \alpha\leq 3, \\
       & \frac{\partial L}{\partial x}\left|_{S^2=0} = \frac{r^2+a^2}{r^2}\widetilde p^3 \right., \\
       & \frac{\partial L}{\partial y}\left|_{S^2=0} = -\frac{r^2+a^2}{r^2}\widetilde p^2\right. .
         \end{align*}
Finally,
\begin{equation*}
 E=\langle\partial_t, \dot \gamma\rangle =\left(-1+\frac{2Mr}{\rho^2}\right)\widetilde p^0+\frac{2Ma}{r\rho^2}\left(-y\widetilde p^2+x\widetilde p^3\right),
\end{equation*}
so that 
 \begin{align*}
         & \frac{\partial E}{\partial \widetilde p^0}\left|_{S^2=0}  = -\frac{\Delta}{r^2+a^2}\right. ,  \\
         & \frac{\partial E}{\partial \widetilde p^\alpha}\left|_{S^2=0}  = 0 \right. \phantom{tafa} \forall \ 1\leq \alpha \leq 3.
 \end{align*}

Making use of the previous calculations, we see that the differential of 
\[h=(\mathrm q, \mathfrak{Q}-E^2 \cdot Q_{\text{trap}}, L  -E\cdot\Phi_{\text{trap}})\]
in terms of coordinates $(r,x,y,\widetilde p^\alpha)$ at an axis point $\{S^2=0\}$ of the photon region contains the submatrix

\begin{equation*}
  \begin{pmatrix}
    \frac{\partial h_j}{\partial \widetilde p^0} &  \frac{\partial h_j}{\partial \widetilde p^2} & \frac{\partial h_j}{\partial \widetilde p^3} & \frac{\partial h_j}{\partial x} &  \frac{\partial h_j}{\partial y}  \\
   
  \end{pmatrix}\left|_{S^2=0}\right. 
  \end{equation*}
  \begin{equation*}
  = \begin{pmatrix}
     -2\frac{\Delta}{r^2+a^2}\widetilde p^0  & 
  2 \frac{r^2+a^2}{r^2}\widetilde p^2 & 2 \frac{r^2+a^2}{r^2}\widetilde p^3 & * & *\\
   -2\frac{\Delta}{r^2+a^2} r^2\widetilde p^0 + \heartsuit & 
  2 \left(r^2+a^2\right)\widetilde p^2 & 2 \left(r^2+a^2\right)\widetilde p^3 & * & *\\
  0&0&0&  \frac{r^2+a^2}{r^2}\widetilde p^3& -\frac{r^2+a^2}{r^2}\widetilde p^2
   
  \end{pmatrix}
\end{equation*}
with $\heartsuit\defeq  2\Delta \widetilde p^0 \left(1-Q_\text{trap}(r)\frac{\Delta}{\left(r^2+a^2\right)^2}\right)$. 
 
Here we have already made use of the fact that $\Phi_{\text{trap}}$ vanishes at axis points of the photon region (see Section~\ref{intro}). 

First, note that at least one of the partial derivatives $\frac{\partial h_1}{\partial \widetilde p^2}|_{ S^2=0}=2 \frac{r^2+a^2}{r^2}\widetilde p^2 $ 
or $\frac{\partial h_1}{\partial \widetilde p^3}|_{ S^2=0}=2 \frac{r^2+a^2}{r^2}\widetilde p^3$ is non-zero, 
since we can exclude $\widetilde p^2=\widetilde p^3=0$. (Photons on the axis with $\widetilde p^2=\widetilde p^3=0$ are radial in- or outgoers, hence not trapped.) Similarly, the last row of the matrix cannot vanish. 
To see that $h$ is submersive where the photon region intersects the axis (in the tangent bundle), it is enough to show that $\heartsuit\neq 0$ for trapped photons on the axis. Since obviously $2\widetilde p^0\Delta\neq 0$, 
for photons in the DOC, one only needs to show that $\left(1-Q_\text{trap}(r)\frac{\Delta}{r^2+a^2}\right)\neq 0$. 
Since $\Phi=0$ for trapped photons on the axis, we plug $\Phi=0$ into $\frac{R(r)}{E^2}=0$ and solve for $Q$. This gives $Q=\frac{1}{\Delta}\left(r^4+a^2r^2+2Ma^2r\right)$ for every trapped photon on the axis. 
Hence, by $Q= Q_\text{trap}$, we get 
\begin{align*}
 1-Q_\text{trap}(r)\frac{\Delta}{\left(r^2+a^2\right)^2} & = \frac{\Delta a^2}{\left(r^2+a^2\right)^2} >0,
\end{align*}

which makes $\heartsuit$ non-zero for trapped photons on the axis.

We have thus seen that the differential of $h$ is surjective at every axis point of the photon region $P'$ in the tangent bundle. By the submersion theorem we may conclude 
that a neighborhood of the North (similarly: South) Pole of $P'$ is a submanifold of $TK$. 

\end{proof}
From Propositions \eqref{on-axis} and \eqref{off-axis}, we immediatley get the following: 
\begin{thm}\label{submanifold}
 For a subcritical Kerr spacetime $K$, the photon region $P'$ in $TK$ and the photon region $P$ in $T^*K$ are smooth submanifolds.
\end{thm}

\section{The Topology of the Kerr Photon Region in the (Co-)tangent Bundle}\label{sec:topo}

We now turn our attention to the topology of $P'$ and $P$. From now on, it will be more useful to work with $P\subseteq T^*K$. 
We will prove the following theorem towards the end of this section and first show important lemmata and propositions to be used in the proof. 
The claim of Theorem~\ref{topo} follows from Dyatlov's implicit function theorem argument in Section 3 of~\cite{dyatlov}. 

\begin{thm}\label{topo}
 The Kerr photon region and in particular the Schwarzschild photon sphere in the (co-)tangent bundle have topology $SO(3)\times \mathbb R^2$. 
\end{thm}

\begin{proof}\renewcommand{\qedsymbol}{}

Since the photon region in any member of the Kerr family is invariant under time translation $\partial_t$ and under rescaling of energy $\left(E, p^0, p^1, p^2, p^3\right)\mapsto  \lambda\left(E, p^0, p^1, p^2, p^3\right)$, $\lambda >0$, 
we may therefore restrict our attention to a $6$-dimensional  slice 
\begin{equation*}
\{t=0, p_0(=-E)=-1\}
\end{equation*}
in the cotangent bundle and show that the photon region in this slice, 
\begin{equation*}
P_0\defeq P\cap\{t=0, p_0(=-E)=-1\},\end{equation*}
 has topology $SO(3)$.
\end{proof}

In the Schwarzschild case, we see immediately that $P_0$ is the bundle of tangent $1$-spheres to the $2$-sphere $\{r=3M\}$. Such a sphere bundle is, of course, topologically just $T^1\mathbb S^2 $, and it is well-known that $T^1\mathbb S^2$ is homeomorphic to $SO(3)$.

The remainder of this section is dedicated to the proof of the rotating Kerr case. In order to construct explicit bundle charts of $P_0$, we first prove the following

\begin{lem}\label{overliner}
Consider a Kerr spacetime with $0<a<M$.
 \begin{enumerate}
  \item 
There are smooth functions $r_\text{min}, r_\text{max}: \left(0, \pi \right) \rightarrow \mathbb R$ that give, for every $\vartheta\in \left(0, \pi \right)$, the minimal and maximal Boyer--Lindquist radii where trapped photons with latitude $\vartheta$ may be located.
 
\item There is a smooth function $\overline r: \left(0, \pi \right)\times [-1,1]\rightarrow \mathbb R$ with 
\begin{align*}
& \overline r(\vartheta, -1)=r_\text{min}\\
&\overline r(\vartheta,0)=r_m, \text{ and} \\
&\overline r (\vartheta, 1)=r_\text{max}
\end{align*}
for all $\vartheta \in \left(0, \pi \right)$, and with $\overline r \left(\tfrac\pi 2 +\vartheta, \cdot\right)=\overline r \left(\tfrac\pi 2 -\vartheta, \cdot\right)$ for all $\vartheta \in \left(0, \tfrac\pi 2\right)$, 
\end{enumerate}
(where $r_m$ is given by Formula~\eqref{rm}), and with the additional property that 
$\overline r (\vartheta , \cdot)$ has a smooth inverse for every $\vartheta\in(0,\pi)$.
\end{lem}

The function $\overline{r}(\vartheta, \cdot)$ parametrizes the radial width of the crescent moon in Figure~\ref{crescent}. 

\begin{proof}
\begin{enumerate}
 \item 
 The boundary of the region accessible to trapping in the spacetime is exactly where trapped photons have vanishing $\vartheta$-motion ($\dot\vartheta=0$); hence, 
 the wanted functions $r_\text{min}, r_\text{max}$ are given implicitly by the requirements that the right-hand side of the $\Theta$-equation~\eqref{Theta} with $Q=Q_{\text{trap}}(r)$ and $\Phi=\Phi_{\text{trap}}(r)$ plugged in vanishes:
\begin{align*}
0&= Q_{\text{trap}}(r_\text{min}(\vartheta))-\left(\Phi_{\text{trap}}(r_\text{min}(\vartheta))-a^2\right)\cos^2(\vartheta), \\
0&= Q_{\text{trap}}(r_\text{max}(\vartheta))-\left(\Phi_{\text{trap}}(r_\text{max}(\vartheta))-a^2\right)\cos^2(\vartheta)
\end{align*} 

for all $\vartheta\in \left(0, \pi \right)$,
and the additional conditions that $r_\text{min}<r_m<r_\text{max}$ and 
that $r_\text{min}(\vartheta), r_\text{max}(\vartheta)$ are the the zeros of $\Theta (\cdot , \vartheta)=0$ that are closest to $r_m(\vartheta)$ for each $\vartheta$.

Thus, $r_\text{min}$ and $r_\text{max}$ are well-defined by elementary properties of the region accessible to trapping, and smoothness follows from the smooth dependence of $\Theta$ on $r$ and $\vartheta$. 

\item This follows directly from the first part of the lemma, for example by explicitly defining 
\begin{equation*}\overline r (\vartheta, s)\defeq 
\begin{cases} 
\left(r_m(\vartheta )-r_{\text{min}}(\vartheta )\right)s+r_m(\vartheta ) & \text{ for } s<0, \\ 

r_m(\vartheta )\left(\frac{r_{\text{max}}(\vartheta )+r_{\text{min}}(\vartheta )-r_m(\vartheta )}{r_m(\vartheta)}\right)^{s^2} +\left(r_m(\vartheta )-r_{\text{min}}(\vartheta )\right)s & \text{ for } s\geq 0.
\end{cases}
\end{equation*}
\end{enumerate}

\end{proof}

For the following lemma, we fix some notation and interpret the product $\mathbb S^1\times \mathbb S^1$ as follows: let the first factor $\mathbb S^1$ be parametrized by $\varphi\in [- \pi , \pi)$, and the second factor $\mathbb S^1$ 
be viewed as $[-1, 1]\times \{-1, 1\}/\sim_{\mathbb S^1}$, where $\sim_{\mathbb S^1}$ is the equivalence 
relation on $[-1, 1]\times \{-1, 1\}$ generated by identifying $(\pm 1,  1)$ with $(\pm 1, - 1)$. We thus denote elements of $\mathbb S^1\times \mathbb S^1$ in the form $(\varphi, s, \varsigma)$. 

Recall that we write $P_0=P\cap\{t=0, p_0=-1\}$. 

\begin{lem}\label{topoeq}

There are functions $p_2, p_3:  \left(0, \pi \right)\times \mathbb S^1\times \mathbb S^1\rightarrow \mathbb R$ such that the map 
\begin{equation*}
H: \left(0, \pi \right)\times \mathbb S^1\times \mathbb S^1\rightarrow P_0\setminus\{S^2=0\},
\end{equation*}
given in Boyer--Lindquist coordinates in the form 
\begin{equation*}
(\vartheta, \varphi, s, \varsigma)\mapsto H(\vartheta, \varphi, s, \varsigma)=(\overline r(\vartheta, s), \vartheta, \varphi, 0, p_2(\vartheta, \varphi, s, \varsigma), p_3(\vartheta, \varphi, s, \varsigma))
\end{equation*}
is a diffeomorphism, where $\overline r $ is defined as in Lemma \ref{overliner} (and hence only depends on $|\vartheta-\frac\pi 2|$ and $s$). 
\end{lem}

\begin{proof}

We define functions $p^2, p^3: \left(0, \pi \right)\times \mathbb S^1\times \mathbb S^1\rightarrow \R$ as the solutions of Equations~\eqref{3mid} and~\eqref{3end} with $\overline r(\vartheta, s)$ plugged in for $r$ and $E=1$ plugged in for $e$, 
and with the additional requirement that 
$\sgn p^2 (\vartheta, \varphi, s, \varsigma)  =\varsigma$.
In other words, 
\begin{align*}
 p^2(\vartheta, \varphi, s, \varsigma)&=\varsigma\left(\frac{1}{ \rho^2(\overline r, \vartheta)}\cdot\left(Q_{\text{trap}}(\overline r)-\left(\frac{\Phi^2_{\text{trap}}(\overline r)}{S^2}-a^2\right)C^2\right)\right)^{\frac 1 2 },\\
 p^3 (\vartheta, \varphi, s, \varsigma)& = \frac{1}{\Delta (\overline r)\rho^2(\overline r, \vartheta)}\cdot \left(2M\overline r a+\left(\rho^2(\overline r, \vartheta)-2M\overline r\right)\frac{\Phi_{\text{trap}}(\overline r)}{S^2}\right),
\end{align*}
where we use the shorthand $\overline r =\overline r \left(\vartheta, s \right)$. 

Both $p^2$ and $p^3$ are smooth on $ (0, \pi)\times \mathbb S^1\times \mathbb S^1$, 
and they are of course, as the notation suggests, meant to be used as $7$-th and $8$-th coordinate in $TK$.

Then, a function $p^0$ is defined as the unique positive solution to \begin{equation*}g\left((p^0, 0, p^2, p^3), (p^0, 0, p^2, p^3)\right)=0;\end{equation*} it is smooth by smoothness of the metric $g$. 

The functions $p_2, p_3$ from the present lemma are obtained by type-changing the vector $(p^0, 0, p^2, p^3)\in T_{}K$. 

Bijectivity of $H$ is clear by the construction, and smoothness of the inverse map $H^{-1}$ follows directly from the fact that for every $\vartheta\in(0,\pi)$, the radial width function $\overline r(\vartheta, \cdot)$ has a smooth inverse. 

\end{proof}

\begin{lem}\label{convex}
Let $0<\vartheta_0<\pi$ and $I_{\vartheta_0}\defeq\left(\Phi_{\text{trap}}(r_{\text{min}}(\vartheta_0)), \Phi_{\text{trap}} (r_{\text{max}}(\vartheta_0))\right)$. There is a unique smooth function $\overline p_2: I_{\vartheta_0}\rightarrow \R_{\geq 0}, p_3\mapsto \overline p_2(p_3)$ such that there is an element of $P_0$ with latitude $\vartheta_0$, angular momentum $p_3$, and third covariant Boyer--Lindquist coordinate $p_2=\overline p_2(p_3)$. 

If $\sin^2\vartheta_0$ is small enough, $\overline p_2$ is concave.

\end{lem}

\begin{proof}
Recall that the covariant Boyer--Lindquist coordinate $p_3$ coincides with the scaled angular momentum $\Phi=\Phi_{\text{trap}}$ for every photon in $P_0$. 
Since $\Phi_{\text{trap}}$ is strictly decreasing in $r$ on the photon region, we may regard $r$ as a function of $p_3$ and write $r(p_3)$ for the solution of $p_3=\Phi_{\text{trap}}(r)$. 

Similarly to the treatment of the functions $p^2, p^3$ in the proof of Lemma~\ref{topoeq}, we use Equation~\eqref{Thetascale} to see that $\overline p_2>0$ is given by
\begin{equation*}\left(\overline p_2\right)^2=
Q_{\text{trap}}(r(p_3))-\left((p_3)^2-a^2\right)\frac{C^2}{S^2}-a^2\frac{C^4}{S^2}.
\end{equation*}

Therefore, 

\begin{equation*}\frac{d^2}{d(p_3)^2}\left(\overline p_2(r)\right)^2=
\frac{d^2}{d (p_3)^2}Q_{\text{trap}}(r(p_3))-2\frac{C^2}{S^2}.
\end{equation*}

Since $\frac{d^2}{d (p_3)^2}Q_{\text{trap}}(p_3)$ is bounded in a neighborhood of the photon region and $\frac{C^2}{S^2}$ approaches infinity at the poles, $\frac{d^2}{d(p_3)^2}\left(\overline p_2(r)\right)^2$ is negative in a punctured neighborhood of any pole. 
Hence, $(\overline p_2)^2$ and also $\overline p_2$ are concave if $\sin^2\vartheta_0$ is chosen sufficiently small. 
\end{proof}

\begin{lem}\label{toponorth}
There is a $0<\varepsilon<1$ such that the neighborhood $P_0\cap  \{\sin^2\vartheta<\varepsilon\}$ of the North and South Pole in of the photon region in the phase space slice $\{t=0, p_0=-1\}$ of a subcritical Kerr spacetime is 
is diffeomorphic to  $\left(\mathbb S^2\cap \{\sin^2\vartheta<\varepsilon \}\right)\times \mathbb S^1$ via 
\begin{equation*}\Psi: P_0\cap  \{\sin^2\vartheta<\varepsilon\} \rightarrow \left(\mathbb S^2\cap \{\sin^2\vartheta<\varepsilon \}\right)\times \mathbb S^1,
\end{equation*}
 where $\Psi$ is given in the coordinates of the proof of Proposition~\ref{off-axis}
by 
\begin{equation*} \Psi ( 0, r, \vartheta, \varphi, -1, 0, \widetilde p_2, \widetilde p_3)\defeq \left(\vartheta, \varphi,[\operatorname{atan2} (\widetilde p_2, \widetilde p_3)]\right).
\end{equation*}
 (A definition of $\operatorname{atan2}$ is given in the proof below.)

\end{lem}
\begin{proof}

We now view the sphere $\mathbb S^1$ as the quotient of the real line that is obtained by factoring out the equivalence relation generated 
by $u\sim u+2\pi$ and denote its elements by $[u]\defeq \{u+2\pi k: k\in \mathbb Z\}$. 

Note that using $\left(\vartheta, \varphi \right)$ also on the poles causes no problem, even though they are not coordinate functions at the poles.

 As usual, $\operatorname{atan2}: \R^2\setminus\{(0,0)\}\rightarrow \left(-\pi ,  \pi \right]$ gives the angle between the positive $\widetilde p_2$-axis and $(\widetilde p_2, \widetilde p_3)$ (keeping track of its sign) and  is defined by 
 
 $\operatorname{atan2}(v, u)=
  \begin{cases} 
\arctan \left(\frac v u \right) & \text{ for } u> 0, \\ 
\arctan \left(\frac v u \right)+\sgn v \cdot \pi & \text{ for } u< 0 , \\
\sgn v\cdot \frac \pi 2 &\text{ for } u=0.
\end{cases}$

Let $\varepsilon >0$ be such that the function $\overline p_2 (\vartheta, \cdot)$ from Lemma~\ref{convex} is concave for every $\vartheta$ with $\sin^2\vartheta<\varepsilon$. 

 The map $\Psi$ as given in the claim is well-defined and smooth. 
  We write $\Psi_\text{North}$ and $\Psi_\text{South}$ for the restrictions of $\Psi$ to the Northern and Southern pole caps $P_0\cap  \{\sin^2\vartheta<\varepsilon, \vartheta< (>) \frac\pi 2\}$.
 Since the union  of the pole caps is disjoint, 
 it suffices to show that $\Psi_\text{North}$ and $\Psi_\text{South}$ are diffeomorphisms. 

Focusing on $\Psi_\text{North}$, we prove that $\Psi_\text{North}$ is a bijection by arguing that for fixed $ (\vartheta_0, \varphi_0)$, the map 
\begin{equation*}
\Psi_\text{North}( 0, r, \vartheta_0, \varphi_0, -1, 0, \cdot, \cdot): P_0\cap\{ \vartheta =\vartheta_0, \varphi = \varphi_0,
\}\rightarrow\{ \vartheta_0,\varphi_0\}\times \mathbb S^1
\end{equation*}
 is bijective: 
  
 First we fix a $( \vartheta_0, \varphi_0)\in \mathbb S^2$ with $\vartheta_0< \pi $, $\sin^2\vartheta_0<\varepsilon$. 
   
We need to show that $\left[\operatorname{atan2}\right]$ maps the set
\begin{equation*}
\{(\widetilde p_2, \widetilde p_3)\in\R^2: \exists\; r \text{ such that } (0,  r, \vartheta_0, \varphi_0,  -1, 0, \widetilde p_2, \widetilde p_3)\text{ is a trapped photon}\}
\end{equation*}
 bijectively to $\mathbb S^1$. To this end, we change from the covariant coordinates $\widetilde  p_2, \widetilde p_3$ to the ones that belong to Boyer--Lindquist coordinates, $p_2$ and $p_3$. 
The set 
\begin{equation*}\{(p_2, p_3)\in\R^2: \exists\; r \text{ such that } (0,  r, \vartheta_0, \varphi_0,  -1, 0, p_2, p_3)\text{ is a trapped photon}\}
\end{equation*}
bounds a convex region around the origin in $\R^2$, since it can be piecewise parametrized by $p_3\mapsto \left(\pm \overline p_2(p_3), p_3\right)$, where $\overline p_2$ is the concave function from Lemma~\ref{convex}. Its image under the linear bijection 
\begin{equation*}\R^2\ni(p_2, p_3)\mapsto \left(\tan \vartheta_0\cdot (\cos\varphi_0 p_2-\sin\varphi_0 p_3), -\sin\varphi_0 p_2+\cos\varphi_0 p_3\right)\in \mathbb R^2\end{equation*} 
is also a convex set around the origin, which means that $\left[\operatorname{atan2}\right]$ maps it bijectively to $\mathbb S^1$. 

On the other hand, one can easily check that
	\begin{equation*}\operatorname{atan2} (\widetilde p_2, \widetilde p_3) =\operatorname{atan2}\left(\tan \vartheta_0\cdot (\cos\varphi_0 p_2-\sin\varphi_0 p_3), -\sin\varphi_0 p_2+\cos\varphi_0 p_3\right),\end{equation*}
 since $\frac{\widetilde p_2}{\widetilde p_3}$ can be rewritten in terms of $p_2$ and $p_3$ as 
  \begin{equation*}\frac{\widetilde p_2}{\widetilde p_3}=\tan \vartheta_0 \cdot \frac {\cos\varphi_0 p_2-\sin\varphi_0 p_3}{-\sin\varphi_0 p_2+\cos\varphi_0 p_3}\end{equation*}
   if $\widetilde p_3\neq 0$, and $\sgn \widetilde p_2 =\sgn \left(\cos\varphi_0 p_2-\sin\varphi_0 p_3\right)$.
   
This proves that  $\Psi_\text{North}( 0, r, \vartheta_0, \varphi_0, -1, 0, \cdot,\cdot): P_0\cap\{\vartheta=\vartheta_0,\varphi=\varphi_0 \}\rightarrow\{ \vartheta_0, \varphi_0\}\times \mathbb S^1$ is bijective for every 
   $( \vartheta_0, \varphi_0)\in \mathbb S^2$ with $\vartheta_0< \pi$, $\sin^2\vartheta_0<\varepsilon$. 
   
Now we show that $\Psi_\text{North}|_{\{S^2=0\}}$ is a bijection onto ${\{S^2=0\}}\times \mathbb S^1$:
  
  At the North Pole ${\{S^2=0\}}$ of the region accessible to trapping, 
  every lightlike point in the phase space with $\widetilde p_1=0$ is a trapped photon, and it is obvious by the rotational symmetry of the Kerr spacetime that the set 
  $\{(\widetilde p_2, \widetilde p_3) : (-1, 0, \widetilde p_2, \widetilde p_3) \text{ is a photon}\}$ is a circle around the origin in the $\widetilde p_2$-$\widetilde p_3$-plane. As before, it is mapped bijectively under $[ \operatorname{atan2}]$ onto $\mathbb S^1$. 
 
We have now seen that $\Psi_\text{North}$ is bijective. A straightforward calculation yields that the differential of $\Psi_\text{North}$ has full rank everywhere, so that by the implicit function theorem $\Psi_\text{North}$ is a diffeomorphism, which proves the claim. 
\end{proof}

In order to determine the topology of $P_0$ (and thereby the topology of $P$), we will first use the maps $H$ and $\Psi$ that were constructed in the Lemmata~\ref{topoeq} and~\ref{toponorth} to conclude: 
\newpage

\begin{prop} The Kerr photon region slice $P_0$ has first fundamental group $\mathbb Z_2$. \label{funda}
\end{prop}
\begin{proof}
The diffeomorphisms $H$ from Lemma~\ref{topoeq} and  $\Psi$ from Lemma~\ref{toponorth} will be used in what follows. 
Recall (from the proof of Lemma~\ref{on-axis}) the covariant coordinates $\widetilde p_i$ that are naturally associated with the coordinates $(t,r, x=r \cos\varphi\sin\vartheta, y=r \sin\varphi\sin\vartheta)$. 

 While $U_\text{North}\approx B_1(0)\times \mathbb S^1$ and $U_\text{South}\approx B_1(0)\times \mathbb S^1$ both have $\mathbb Z$ as their first fundamental group, $U_\text{Eq}\approx \left(-\pi, \pi \right)\times \mathbb S^1\times \mathbb S^1$ has first fundamental group $\mathbb Z\times \mathbb Z$. 
 Moreover, $\pi_1\left(U_\text{Eq}\cap U_\text{North}\right)=\pi_1\left(U_\text{Eq}\cap U_\text{South}\right)=\mathbb Z\times \mathbb Z$. 

Let $\frac\pi 2<\vartheta_0<\pi$ be such that every point of $P_0$ with latitude $\vartheta_0$ is in the domain of $\Psi_{\text{North}}$. We define homotopies of closed paths $\gamma_{1}, \gamma_{2}: \mathbb S^1\times I\rightarrow P_0$ by 

\begin{align*}
\gamma_{1,\lambda} (\varphi)&\defeq H\left(\vartheta_0 -\lambda(2\vartheta_0-\pi), \varphi , 0,-1\right), \\
\gamma_{2, \lambda} (s, \varsigma)&\defeq H\left( \vartheta_0 -\lambda(2\vartheta_0-\pi), 0, s, \varsigma\right).
\end{align*}

By construction, the paths $\gamma_{1,0}$ and $\gamma_{2,0}$ are representatives of a set of generators for 
$\pi_1\left(U_\text{Eq}\cap U_\text{North}\right)$, and the analogous statement holds for $\gamma_{1,1}$ and $\gamma_{2,1}$ in $\pi_1\left(U_\text{Eq}\cap U_\text{South}\right)$, as well as for  $\gamma_{1,\lambda}$ and $\gamma_{2,\lambda}$ in $\pi_1\left(U_\text{Eq}\right)$
(for every $\lambda$). 

We will now determine what elements in $\pi_1(U_{\text{North}})$ the paths $\gamma_{i,0}$ represent and what elements in  $\pi_1(U_{\text{South}})$ the paths $\gamma_{i,1}$ represent (for $i=1,2$). 

Note that by construction of $H$, the image of each $\gamma_{1,\lambda}$ consist of points with $r=r_m$ and hence with $p_3=0$. 

This makes it easy to calculate---using the standard transformations for the coordinates of the phase part of $TK^*$---that for points in the image of a path $\gamma_{1,\lambda}$ with $\lambda\neq \frac 1 2$: 
\begin{align}
\widetilde p_2\circ{\gamma_{1,\lambda}} (\varphi)&=\frac{p_2}{r\cos \vartheta}\cos \varphi \label{pzwo}\\
\widetilde p_3\circ{\gamma_{1,\lambda}} (\varphi)&=-\frac{p_2}{r\sin\vartheta}\sin\varphi,\label{pzwozwo} 
\end{align}
where we use $\vartheta$ as shorthand for $\vartheta(\lambda)=\vartheta_0 -\lambda(2\vartheta_0-\pi)$ and $p_2$ as shorthand for $p_2\circ{\gamma_{1,\lambda}}(\varphi)$.

Clearly, $|p_2|\circ{\gamma_{1,\lambda}} (\varphi)$ can be calculated from $p_3\circ{\gamma_{1,\lambda}}=0$, $r\circ{\gamma_{1,\lambda}}=r_m$, and the modulus of the latitude, $|\vartheta-\frac\pi 2|$. 
Since we have chosen negative sign of $p_2$ for all points in the range of $\gamma_{1,\lambda}$, even
$p_2\circ{\gamma_{1,\lambda}} (\varphi)$ only depends on $|\vartheta-\frac\pi 2|$.

This allows us to simplify \eqref{pzwo}--\eqref{pzwozwo} to 
\begin{align*}
\widetilde p_2\circ{\gamma_{1,0}} (\varphi)&=A\cos \varphi,\\
\widetilde p_2\circ{\gamma_{1,1}} (\varphi)&=-A\cos\varphi,\\
\widetilde p_3\circ{\gamma_{1,\lambda}}(\varphi) &=B\sin\varphi  \hspace{.3cm}\text{for } \lambda =0,1\\
\end{align*}
for positive constants $A, B$. 

We see from these formulas that the map $P_3\circ\Psi  \circ\gamma_{1,0}: \mathbb S^1\rightarrow \mathbb S^1$ (where $P_3$ is the projection onto the third component), 
given by $[\operatorname{atan2}(\widetilde p_2\circ{\gamma_{1,0}}, \widetilde p_3\circ{\gamma_{1,0}})]$, has index~$1$ and is hence homotope to the identity map on $\mathbb S^1$. 

Similarly, $P_3\circ\Psi  \circ\gamma_{1,1}: \mathbb S^1\rightarrow \mathbb S^1$ is of index~$-1$ and can be thought of as the map $\operatorname{-id}: \mathbb S^1\rightarrow \mathbb S^1$.

We may thus note for later use \textbf{Fact 1}: 
the path $\Psi\circ\gamma_{1,0}$ is equivalent to $\varphi\mapsto (\vartheta_0, 0, \varphi)$ in $\pi_1(\Psi(U_{\text{North}}\cap U_\text{Eq}))$, and the path $\Psi\circ\gamma_{1,1}$ 
is equivalent to $\varphi\mapsto (\pi-\vartheta_0, 0, -\varphi)$ in $\pi_1(\Psi(U_{\text{South}}\cap U_\text{Eq}))$.

Similarly as we just did for $\gamma_{1,\lambda}$, we now calculate the components $\widetilde p_2$ and $\widetilde p_2$ on the orbits of $\gamma_{2,\lambda}$ for $\lambda\neq \frac 1 2 $: 
\begin{align*}
\widetilde p_2\circ{\gamma_{2,\lambda}}(s, \varsigma) &=\frac{1}{r\cos \vartheta} p_2\\
\widetilde p_3\circ{\gamma_{2,\lambda}}(s, \varsigma) &=\frac{1}{r\sin\vartheta}p_3. 
\end{align*}
Here, $p_2$ and $p_3$ depend on $\vartheta(\lambda)= \vartheta_0 -\lambda(2\vartheta_0-\pi)$ and $(s, \varsigma)$. 
Actually, $p_2$ and $p_3$ are independent of the sign of $\vartheta-\frac \pi 2$, since the function $\overline r$ used in the construction of $H$ is. 

This allows us to write \begin{align}
\widetilde p_2\circ{\gamma_{2,0}} (s, \varsigma)&=- \widetilde A p_2,\label{pezwo}\\
\widetilde p_2\circ{\gamma_{2,1}} (s, \varsigma)&=\widetilde A p_2,\\
\widetilde p_3\circ{\gamma_{2,\lambda}}(s, \varsigma) &=\widetilde B p_3  \hspace{.3cm}\text{for } \lambda =0,1 \label{pezwozwo} 
\end{align}
for positive constants $\widetilde A, \widetilde B$. 

We already know that $\gamma_{2,0}$ represents one of the generators of 
$\pi_1\left(U_\text{Eq}\cap U_\text{North}\right)$, and that $\gamma_{2,1}$ represents one of the generators of
$\pi_1\left(U_\text{Eq}\cap U_\text{South}\right)$. 
Choosing an orientation for $\mathbb S^1\equiv [-1, 1]\times \{-1, 1\}/\sim_{\mathbb S^1}$, we can decide that $P_3\circ\Psi\circ\gamma_{2,i}: \mathbb S^1\rightarrow \mathbb S^1$ has index~$1$. 
Then $P_3\circ\Psi\circ\gamma_{2,i}: \mathbb S^1\rightarrow \mathbb S^1$ has index~$-1$.
We can now state \textbf{Fact 2}: 
the path $\Psi\circ\gamma_{2,0}$ is equivalent to $\varphi\mapsto (\vartheta_0, 0, \varphi)$ in $\pi_1(\Psi(U_{\text{North}}\cap U_\text{Eq}))$, 
and the path $\Psi\circ\gamma_{1,1}$ is equivalent to $\varphi\mapsto (\pi-\vartheta_0, 0, -\varphi)$ in $\pi_1(\Psi(U_{\text{South}}\cap U_\text{Eq}))$. 

In the last step of the proof, we finally calculate $\pi_1(P_0)$ using the Seifert--van Kampen theorem. 
Combining Fact 1 and Fact 2, we note that in ${\pi_1\left(U_\text{North}\cap U_\text{Eq}\right)}$, 
\begin{equation*}\left[\gamma_{1,0}\right]_{\pi_1\left(U_\text{North}\cap U_\text{Eq}\right)}=\left[\gamma_{2,0}\right]_{\pi_1\left(U_\text{North}\cap U_\text{Eq}\right)}, \end{equation*}

and in ${\pi_1\left(U_\text{South}\cap U_\text{Eq}\right)}$, that

\begin{equation*}\left[\gamma_{1,1}\right]_{\pi_1\left(U_\text{South}\cap U_\text{Eq}\right)}=\left[\gamma_{2,1}\right]_{\pi_1\left(U_\text{South}\cap U_\text{Eq}\right)}. 
\end{equation*}

We apply the Seifert--van Kampen theorem two times; first to $U_\text{North}$ and $U_\text{Eq}$, then to $U_\text{North}\cup U_\text{Eq}$ and $U_\text{South}$). 

For the sake of readability, the homomorphisms of the different fundamental groups induced by the various inclusion maps are notationally omitted. 

Using the group presentations 
\begin{equation*}
 \pi_1\left(U_\text{North}\right)=\langle\left[\gamma_{1,0}\right]_{\pi_1\left(U_\text{North}\right)} , \left[\gamma_{2,0}\right]_{\pi_1\left(U_\text{North}\right)}\rangle
 \end{equation*}
and 
\begin{equation*}
\pi_1\left(U_\text{Eq}\right)=\langle\left[\gamma_{1,0}\right]_{\pi_1\left(U_\text{Eq}\right)} , \left[\gamma_{2,0}\right]_{\pi_1\left(U_\text{Eq}\right)}\rangle,
\end{equation*}

one obtains $\pi_1\left(U_\text{Eq}\cup U_\text{North}\right)$ by factoring the identifications

\begin{align*}
\left[\gamma_{1,0}\right]_{\pi_1\left(U_\text{North}\right)}&=\left[\gamma_{2,0}\right]_{\pi_1\left(U_\text{Eq}\right)}, \\
\left[\gamma_{1,0}\right]_{\pi_1\left(U_\text{North}\right)}&=\left[\gamma_{1,0}\right]_{\pi_1\left(U_\text{Eq}\right)}
\end{align*}
out of the group 
\begin{equation*}
\langle\left[\gamma_{1,0}\right]_{\pi_1\left(U_\text{North}\right)} , \left[\gamma_{2,0}\right]_{\pi_1\left(U_\text{North}\right)}, \left[\gamma_{1,0}\right]_{\pi_1\left(U_\text{Eq}\right)}, \left[\gamma_{2,0}\right]_{\pi_1\left(U_\text{Eq}\right)} \rangle.
\end{equation*}

Now,  $\pi_1(P_0)=\pi_1\left(\left(U_\text{Eq}\cup U_\text{North}\right)\cup U_\text{South} \right)$ can be calculated by factoring the identifications 
\begin{align*}
 \left[\gamma_{1,1}\right]_{\pi_1\left(U_\text{South}\right)}&=\left[\gamma_{2,0}\right]_{\pi_1\left(U_\text{Eq}\right)}, \\
 \left[\gamma_{1,1}\right]_{\pi_1\left(U_\text{South}\right)}&=\left[\gamma_{1,0}\right]_{\pi_1\left(U_\text{Eq}\right)}, 
 \end{align*}
 (which come from Facts 1 and 2) out of the free product of $\pi_1\left(U_\text{Eq}\cup U_\text{North}\right)$ with $\pi_1\left(U_\text{South} \right)$.
 
Hence, $\pi_1(P_0)$ is the quotient of $\langle \left[\gamma_{1,0}\right]_{\pi_1\left(U_\text{Eq}\right)}, \left[\gamma_{2,0}\right]_{\pi_1\left(U_\text{Eq}\right)}, 
 \left[\gamma_{1,0}\right]_{\pi_1\left(U_\text{North}\right)}, \left[\gamma_{1,1}\right]_{\pi_1\left(U_\text{South}\right)}\rangle $
 after factoring out the identifications
\begin{align*}
\left[\gamma_{1,0}\right]_{\pi_1\left(U_\text{North}\right)}&=\left[\gamma_{2,0}\right]^{-1}_{\pi_1\left(U_\text{Eq}\right)}, \\
\left[\gamma_{1,0}\right]_{\pi_1\left(U_\text{North}\right)}&=\left[\gamma_{1,0}\right]_{\pi_1\left(U_\text{Eq}\right)},\\
  \left[\gamma_{1,1}\right]_{\pi_1\left(U_\text{South}\right)}&=\left[\gamma_{2,0}\right]_{\pi_1\left(U_\text{Eq}\right)}, \\
 \left[\gamma_{1,1}\right]_{\pi_1\left(U_\text{South}\right)}&=\left[\gamma_{1,0}\right]_{\pi_1\left(U_\text{Eq}\right)}.
\end{align*}
The quotient we are left with is the group $\mathbb Z_2$. 
\end{proof}

We are now going to conclude in a very short argument that the photon region in the phase space of a subcritical Kerr spacetime has topology $SO(3)\times \mathbb R^2$:

\begin{proofof}{Theorem~\ref{topo}} Since $P_0$ is a closed $3$-dimensional manifold with $\pi_1(P_0)=\mathbb Z_2$, it is doubly covered by an $\mathbb S^3$ (by the Poincar\'e conjecture). 
By the elliptization conjecture, this $\mathbb S^3$ can be taken to be the standard $3$-sphere and the group $\mathbb Z_2$ as a subgroup of $SO(3)$ acting on it. (For the statements of the Poincar\'e and the elliptization conjecture, see e.g.~\cite{thurston, scott}; for 
the proofs covering these conjectures see~\cite{perel1, perel2, perel3}.) Hence, $P_0$ is the quotient $\mathbb S^3 /\mathbb Z_2 \approx \mathbb R P^3\approx SO(3)$. 

Recalling how $P_0$ was obtained as a slice $P\cap \{t=0, p_0=-1\}$ of the photon region in the phase space, this proves Theorem~\ref{topo}. 

\end{proofof}

\section{Outlook}

We have described how the phenomenon of trapping of light in subcritical Kerr spacetimes can be better understood in the framework of the cotangent bundle and have characterized the set of (future) trapped photons as a submanifold of $T^*K$ with topology $SO(3)\times\R^2$. 
It is natural to ask if analogous results can be obtained with similar methods for other stationary spacetimes, possibly involving a cosmological constant. Note that in~\cite{dyatlov}, the implicit function argument for the computation of the topology of the photon sphere in the phase space is also applied to the Kerr--de Sitter spacetime.
 
Null geodesics of constant coordinate radius in (various subfamilies of) the Pleba\'{n}ski--Demia\'{n}ski class of metrics have been studied in~\cite{gpl14, gplaccel}. To our knowledge, it has not yet been rigorously investigated whether these photon orbits are the only trapped photons in the respective spacetimes, which would be a crucial first step for a geometric understanding of the photon region. 
Results for spherical photons in the Pleba\'{n}ski--Demia\'{n}ski spacetime family show a very similar spacetime picture of trapping as in the Kerr family (\cite{gpl14, gplaccel}); one might also hope to prove in this setting that the photon region in the 
phase space is a submanifold with topology $ SO(3)\times \R^2$, but of course, the calculations would become much more involved in this generalized setting. 

We can also ask similar questions about the photon region in stationary spacetimes of higher dimensions. By the symmetry of a Schwarzschild--Tangherlini black hole~\cite{Tangherlini} of dimension $n+1$ and mass $M$, it is known that 
trapping of light occurs at a fixed radius $r=nM$ for every photon that is initially tangent to the hypersurface $\{r=nM\}$. Hence (by arguments just like in the $4$-dimensional Schwarzschild case), the photon region in the phase space can be expected to be a
submanifold with the topology of a tangent unit sphere bundle $ T^1\mathbb S^{n-1}$, times $\mathbb R^2$. 

It seems reasonable to conjecture that for Myers--Perry spacetimes of dimenion $n+1$, the photon region in the phase space has the same topology and bundle structure as the one of the Schwarzschild--Tangherlini solution of dimension $n+1$ (
 whether the various rotation parameters coincide or not), since letting the rotational parameters tend to zero should not cause jumps in the topology, just as in the $3+1$-dimensional Kerr case. 
This can, however, not be proved by similar methods as the ones we used for the $4$-dimensional case, since our proof relies on the classification of $3$-manifolds via their fundamental groups, 
and there is no similar classification available for higher dimenions. 

The presented way to determine the topology of the Kerr photon region in phase space might turn out to be useful in proving a uniqueness theorem for asymptotically flat, stationary, vacuum spacetimes with a photon region inner boundary, in the spirit of the static results in this direction discussed in the introduction. 
 
\section{Acknowledgements}

We would like to thank Pieter Blue and Andr\'as Vasy for useful comments, as well as Greg Galloway for suggesting a simplification of the concluding argument. 
Furthermore, we thank Oliver Sch\"on for generating the figures for this article. 

This work is supported by the Institutional Strategy of the University of T\"ubingen (Deutsche Forschungsgemeinschaft, ZUK 63).

\bibliographystyle{plain}

\begin{thebibliography}{10}

\bibitem{carla}
Carla Cederbaum.
\newblock Uniqueness of photon spheres in static vacuum asymptotically flat
  spacetimes.
\newblock In {\em Complex Analysis \& Dynamical Systems IV}, volume 667 of {\em
  Contemp. Math.}, pages 86--99. AMS, 2015.

\bibitem{carlagregelectro}
Carla Cederbaum and Gregory~J. Galloway.
\newblock Uniqueness of photon spheres in electro-vacuum spacetimes.
\newblock {\em Class. Quantum Grav.}, 33(7):075006, 2016.

\bibitem{carlagregpmt}
Carla Cederbaum and Gregory~J. Galloway.
\newblock Uniqueness of photon spheres via positive mass rigidity.
\newblock {\em Commun. Anal. Geom.}, 25(2):303--320, 2017.

\bibitem{dafermos}
Mihalis Dafermos and Igor Rodnianski.
\newblock Lectures on black holes and linear waves.
\newblock arXiv:0811.0354 [gr-qc], 2008.

\bibitem{dyatlov}
Semyon Dyatlov.
\newblock Asymptotics of linear waves and resonances with applications to black
  holes.
\newblock {\em Communications in Mathematical Physics}, 335(3):1445--1485, May
  2015.

\bibitem{perel1}
Grisha Perelman.
\newblock The entropy formula for the {R}icci flow and its geometric
  applications, 2002.

\bibitem{perel3}
Grisha Perelman.
\newblock Finite extinction time for the solutions to the {R}icci flow on
  certain three-manifolds, 2003.

\bibitem{perel2}
Grisha Perelman.
\newblock Ricci flow with surgery on three-manifolds, 2003.

\bibitem{thurston}
William Thurston.
\newblock The geometry and topology of three-manifolds.
\newblock Princeton lecture notes on geometric structures on 3-manifolds, 1980.



\bibitem{grenze}
Arne Grenzebach.
\newblock {\em The Shadow of Black Holes. An analytic description}.
\newblock Springer, Berlin, 2016.

\bibitem{gpl14}
Arne Grenzebach, Volker Perlick, and Claus L{\"a}mmerzahl.
\newblock Photon regions and shadows of {K}err--{N}ewman--{NUT} black holes
  with a cosmological constant.
\newblock {\em Phys. Rev.}, D89(12):124004, 2014.

\bibitem{gplaccel}
Arne {Grenzebach}, Volker {Perlick}, and Claus {L{\"a}mmerzahl}.
\newblock {Photon regions and shadows of accelerated black holes}.
\newblock {\em International Journal of Modern Physics D}, 24:1542024, June
  2015.

\bibitem{hatch3}
Allen Hatcher.
\newblock Notes on basic 3-manifold topology.
\newblock \url{http://pi.math.cornell.edu/~hatcher/3M/3Mdownloads.html}, 2007.
\newblock last accessed 06/18/2018.

\bibitem{oneill}
Barrett O'Neill.
\newblock {\em The Geometry of {K}err Black Holes}.
\newblock Dover Publications, Mineola, New York, 2014.

\bibitem{orlik}
Peter Orlik.
\newblock {\em Seifert manifolds}, volume 291 of {\em Lecture notes in
  mathematics}.
\newblock Springer, Berlin, Heidelberg, New York, 1972.

\bibitem{claudio}
Claudio Paganini, Blazej Ruba, and Marius~A. Oancea.
\newblock Characterization of null geodesics on {K}err spacetimes.
\newblock arXiv:1611.06927 [gr-qc], 2016.

\bibitem{Perlick2004}
Volker Perlick.
\newblock Gravitational lensing from a spacetime perspective.
\newblock {\em Living Reviews in Relativity}, 7(1):9, Sep 2004.

\bibitem{perlick}
Volker Perlick.
\newblock On totally umbilic submanifolds of semi-{R}iemannian manifolds.
\newblock {\em Nonlinear Analysis}, 63(5-7), 2005.
\newblock arXiv:gr-qc/0512066.

\bibitem{scott}
Peter Scott.
\newblock The geometries of 3-manifolds.
\newblock {\em Bulletin of the London Mathematical Society}, 15(5):401--487,
  1983.

\bibitem{shoom1}
Andrey~A. Shoom.
\newblock Metamorphoses of a photon sphere.
\newblock {\em Phys. Rev. D}, 96:084056, Oct 2017.

\bibitem{shoom2}
Andrey~A. Shoom, Cole Walsh, and Ivan Booth.
\newblock Geodesic motion around a distorted static black hole.
\newblock {\em Phys. Rev. D}, 93:064019, Mar 2016.

\bibitem{sommers}
Paul Sommers.
\newblock On {K}illing tensors and constants of motion.
\newblock {\em Journal of Mathematical Physics}, 14(6):787--790, 1973.

\bibitem{Tangherlini}
Frank~R. Tangherlini.
\newblock {Schwarzschild field in $n$ dimensions and the dimensionality of
  space problem}.
\newblock {\em Nuovo Cim.}, 27:636--651, 1963.

\bibitem{teo}
Edward Teo.
\newblock Spherical photon orbits around a {K}err black hole.
\newblock {\em General Relativity and Gravitation}, 35(11):1909--1926, 2003.

\bibitem{yaza}
Stoytcho Yazadjiev and Boian Lazov.
\newblock Uniqueness of the static {E}instein--{M}axwell spacetimes with a
  photon sphere.
\newblock {\em Class. Quantum Grav.}, 32(16):165021, 2015.

\bibitem{yaza2}
Stoytcho~S. Yazadjiev.
\newblock {Uniqueness of the static spacetimes with a photon sphere in
  Einstein-scalar field theory}.
\newblock {\em Phys. Rev. D}, 91(12):123013, 2015.

\bibitem{yaza3}
Stoytcho~S. Yazadjiev and Boian Lazov.
\newblock Classification of the static and asymptotically flat
  {E}instein-{M}axwell-dilaton spacetimes with a photon sphere.
\newblock {\em Phys. Rev. D}, 93(8):083002, 11, 2016.

\bibitem{yoshino}
Hirotaka Yoshino.
\newblock Uniqueness of static photon surfaces: Perturbative approach.
\newblock {\em Phys. Rev. D}, 95:044047, Feb 2017.

\end{thebibliography}
\end{document}